\documentclass[12pt]{amsart}
\usepackage[letterpaper,margin=1.1in]{geometry}

\usepackage{graphicx,color,amsmath,amssymb,mathrsfs}
\usepackage{amsfonts,amsthm,epsfig,epstopdf, esint,soul, mathtools}
\usepackage[mathscr]{euscript}

\newtheorem{theorem}{Theorem}[section]
\newtheorem{lemma}[theorem]{Lemma}
\newtheorem{proposition}[theorem]{Proposition}
\newtheorem{corollary}[theorem]{Corollary}
\newtheorem{remark}[theorem]{Remark}
\newtheorem{definition}[theorem]{Definition}

\usepackage[utf8]{inputenc}
\usepackage[english]{babel}

\usepackage{xcolor}
\usepackage[colorlinks = true,
            linkcolor = blue,
            urlcolor  = blue,
            citecolor = blue,
            anchorcolor = blue]{hyperref}
\newcommand{\R}{{\mathbb{R}}}
\newcommand{\e}{{\epsilon}}
\newcommand{\g}{{\gamma}}

\newcommand{\M}{{\mathcal{M}}}

\newcommand{\vphi}{\varphi}

\newcommand{\na}{{\nabla}}
\newcommand{\B}{\mathcal{B}}

\newcommand{\Ric}{{\rm{Ric}}}

\newcommand{\vol}{{\rm{vol}}}
\newcommand{\ASp}{{\rm{AS}}_p}

\begin{document}

\title[Quantitative Yamabe]{Quantitative Stability for Minimizing Yamabe Metrics}
\author{Max Engelstein}
\author{Robin Neumayer}
\author{Luca Spolaor}
  \thanks{ 
}
\address{Department of Mathematics, University of Minnesota, Minneapolis, MN, 55455, USA.}
\email{mengelst@umn.edu}
\address{Department of Mathematics, Carnegie Mellon University, Pittsburgh, PA, 15213, USA.}
\email{neumayer@cmu.edu}
\address{Department of Mathematics, UC San Diego, La Jolla, CA, 92093, USA.}
\email{lspolaor@ucsd.edu}
\date{\today}

\begin{abstract}
On any closed Riemannian manifold of dimension $n\geq 3$, we prove that if a function nearly minimizes the Yamabe energy, then the corresponding conformal metric is close, in a quantitative sense, to a minimizing Yamabe metric in the conformal class. Generically, this distance is controlled quadratically by the Yamabe energy deficit. Finally, we produce an example for which this quadratic estimate is false.
\end{abstract}

\maketitle

\section{Introduction}\label{s:intro}
Let $(M^n,g)$ be a closed Riemannian manifold of dimension $n\geq 3$. The Yamabe problem consists of finding a metric $\tilde{g}$, conformal to $g$, such that the scalar curvature of $\tilde{g}$ is constant. Given a metric $\tilde{g}$ conformal to $g$, i.e.  $\tilde{g} = u^{4/(n-2)} g$ for  a smooth positive function $u$ on $M$, the scalar curvature, $R_{\tilde{g}}$, of $\tilde{g}$ is given in terms of $u$ and the scalar curvature, $R_g$, of $g$  by
\begin{equation}\label{eqn: comformal change}
R_{\tilde{g}} = u^{1-2^*}\left(-c_n \Delta u + R_g u \right),
\end{equation}
where $2^* = 2n/(n-2)$ and $c_n =4(n-1)/(n-2).$ In particular,  a metric ${\tilde{g}} = u^{4/(n-2)} g$ is a  solution to the Yamabe problem if and only if $u$ is a  smooth positive critical point of the associated energy functional 
\begin{equation}\label{eqn: Yamabe functional}
Q(u) =
\ \frac{ \int_M c_n |\na u|^2 + R_g u^2 \, d\vol_g
}{\left(\int_M u^{2^*}\, d\vol_g\right)^{2/2^*}} =\frac{\int_M R_{\tilde{g}}\, d\vol_{\tilde{g}} }{
\vol_{\tilde{g}}(M )^{2/2^*}
}.
 \end{equation}

  The solution to the Yamabe problem was given by the combined works of Yamabe \cite{OGYamabe}, Trudinger  \cite{Trudinger},  Aubin \cite{AubinYamabe}, and Schoen \cite{SchYamabe} (see also the  survey paper \cite{LeeParker}), which established the existence of a smooth positive minimizer of \eqref{eqn: Yamabe functional}, i.e. a positive function $u \in C^{\infty}(M)$ with $Q(u) = Y(M,[g])$, where we define the Yamabe constant of $(M,g)$ by
  \[
 Y(M,[g]) = \inf\{ Q(u): u \in W^{1,2}(M)\,,\,u\geq 0\}.
 \]
Here $[g]$ denotes the conformal class of $g$. Sometimes, when it won't cause confusion, we will omit the dependence on $M$ and $[g]$.  
 
The Yamabe constant, $Y(S^n, [g_0])$, on the round sphere plays an important role in  the solution to the Yamabe problem on a general manifold, $M^n$. When $Y(M^n,[g]) <Y(S^n, [g_0])$, the existence of a minimizer can be established through analytic methods, either by approximating the Euler-Lagrange equation assiciated to \eqref{eqn: Yamabe functional} by subcritical equations (\cite{Trudinger, LeeParker}), or via concentration compactness methods (see \cite{Lions} or Uhlenbeck, c.f  \cite{LeeParker}). On the other hand,  Aubin \cite{AubinYamabe} and Schoen \cite{SchYamabe} showed that $Y(M^n,[g]) <Y(S^n, [g_0])$  for any closed Riemannian manifold, $(M^n,g)$, that is not conformally equivalent to the round sphere. 

In the case of the round sphere, the class of minimizers $\mathcal{M}_{(S^n,g_0)}$ of \eqref{eqn: Yamabe functional} were explicitly characterized by Aubin \cite{AubinSobolev} and Talenti \cite{Talenti} (see also Obata \cite{Obata}): after composing with a stereographic projection, which maps the problem to Euclidean space, the set of minimizers is exactly the function $v_0 = (1+|x|^2)^{(2-n)/2}$, along with its translations, dilations, and constant multiples on $\R^n$.  
 
In \cite{brezis1985sobolev}, Brezis and Lieb raised the question of quantitative stability for minimizers of the Yamabe functional on the sphere, asking whether the energy deficit $Q_{(S^n,g_0)}(u) -Y(S^n, [g_0])$ of a given function $u \in W^{1,2}(S^n)$ controls its distance to the family of  minimizers $\mathcal{M}_{(S^n,g_0)}$. An optimal solution was given in \cite{BianchiEgnell91}, where Bianchi and Egnell showed that there exists a dimensional constant  $c$ such that 
\begin{equation}\label{eqn: BE}
Q_{(S^n,g_0)}(u) -Y(S^n, [g_0]) \geq c  \left(\frac{\inf \left\{ \| u-v\|_{W^{1,2}(S^n)}\ | \  {v \in \mathcal{M}_{(S^n,g_0)}} \right\}}{\| u \|_{W^{1,2}(S^n )} }\right)^2
\end{equation}
 for any nonnegative $u \in W^{1,2}(S^n)$.\footnote{We note that question of Brezis and Lieb and the result in \cite{BianchiEgnell91} are stated on Euclidean space, but the form \eqref{eqn: BE} follows after composition with stereographic projection and integration by parts.}  This result is sharp in the sense that the exponent $2$ cannot be replaced by a smaller one and the $W^{1,2}$ norm measuring the distance of $u$ to the family of minimizers cannot be replaced by a stronger norm.

 In this paper, we address this question of Brezis and Lieb in the setting of the Yamabe functional on {\it any} smooth closed Riemannian $n$-manifold $(M^n,g)$, with $n\geq 3.$ In contrast to the case of the round sphere, the minimizers for a general manifold are not known in any explicit form.

Fix a closed Riemannian manifold, $(M^n,g)$, of dimension $n\geq 3,$ and let $\M \subset W^{1,2}(M)$ denote the set of all minimizers of $Q(u)$.     Define
 \begin{equation}\label{eqn: dist def}
 d(u,\M) = \frac{\inf\left\{ \| u-v\|_{W^{1,2}(M)} \ | \ v \in \M\right\}}{\| u\|_{W^{1,2}(M)}}.
 \end{equation}
  Notice that the normalization in this definition guarantees that $d(u , \mathcal{M}) \leq 1$ for any $u \in W^{1,2}(M).$ 
  Our first main result is a quantitative stability estimate for minimizers of the Yamabe functional.
 \begin{theorem}[Quantitative stability for minimizers]\label{thm: minimizers}
 	Let $(M^n,g)$ be a $C^\infty$ closed Riemannian manifold of dimension $n\geq 3$ that is not conformally equivalent to the round sphere. There exist constants $c>0$ and  $\g\ge 0 $, depending on $(M,g)$, such that 
 	\begin{equation}\label{e: minimizers}
 	Q(u) - Y(M,[g]) \geq c\, d(u, \mathcal{M})^{2+\g}\, \qquad \forall u \in W^{1,2}(M;\R_+)  \,.
 	\end{equation}
 	Moreover, there exists an open dense subset $\mathcal G$ in the $C^2$ topology on the space of $C^\infty$-conformal classes of metrics on $M$ such that if $[g] \in \mathcal G$, we may take $\g=0$.
 \end{theorem}

From a geometric point of view, one drawback of Theorem \ref{thm: minimizers} is that the distance, $d(u, \mathcal{M})$, depends on the choice of background metric, $g \in [g]$. However, as a consequence of Theorem~\ref{thm: minimizers}, we obtain the following conformally invariant stability estimate.  Define the following conformally invariant distance between two metrics in a conformal class: 
 \[ \| g_u - g_v\| = \left( \int_M |u-v|^{2^*} \, d\vol_g \right)^{1/2^*},
 \]  
where here and in the sequel we will freely make the identification of a conformal metric $g_u = u^{4/(n-2)}g$ and its conformal factor $u$. Although $\| - \|$ is defined with respect to a fixed conformal representative $g \in [g]$, we will show that it is independent of this choice.
Similarly,  in the case when $Y = Y(M, [g])\geq 0$, may define 
\[
\| g_u -g_v\|_* = \left( \int_M c_n |\na u-\na v|^2 + Y (u-v)^2 \, d\vol_g\right)^{1/2}  
\]
for any $ g \in \mathcal{M}(M, g)$ with $\vol_g(M)=1$. Again, although $\| -\|_*$ is  defined with respect to a fixed conformal representative, we show that the definition is independent of this choice.

 \begin{corollary}[Conformal quantitative stability] \label{cor: conformal CONQUEST}
 Let $(M^n,g)$ be a $C^\infty$ closed Riemannian manifold of dimension $n\geq 3$. There exist constants $c>0$ and  $\g\ge 0 $, depending on $M$ and $[g]$, such that 
 	\begin{equation}\label{e: minimizers conf invar}
 	\mathcal{R}_g  - Y(M,[g]) \geq c\, \left(\frac{\inf\{ \| g - \tilde g\| : \tilde{g} \in \M \}}{\vol_g(M)^{1/2^*}}\right)^{2+\g}\, \qquad \forall g \in [g]\,.
 	\end{equation}
 	Here $\mathcal{R}_g=  \vol_g(M)^{-2/2^*}\,\int_M R_{g} \, d\vol_g$ is the volume-normalized total scalar curvature of $g$.
 When $Y= Y(M,[g]) \geq 0$ and $\mathcal{R}_g  - Y(M,[g])  \leq 1$, there exist constants $c>0$ and  $\g\ge 0 $ depending on $M$ and $[g]$ such that 
	\begin{equation}\label{e: minimizers conf invar2}
 	\mathcal{R}_g  - Y(M,[g]) \geq c\, \left(\frac{\inf\{ \| g - \tilde g\|_* : \tilde{g} \in \M \}}{\vol_g(M)^{ 1/2^*}}\right)^{2+\g}\, \qquad \forall g \in [g]\,.
 	\end{equation}
Moreover, for an open dense subset in the $C^2$ topology on the space of conformal classes of $C^\infty$  metrics on $M$, we may take $\g=0$. 
 \end{corollary}

\begin{remark}
	{\rm
	Observe that the denominator in \eqref{eqn: dist def} is the $W^{1,2}$ norm of $u$, not the $L^{2^*}$ norm of $u$. This normalization is due to the different scalings of the left- and right-hand sides of \eqref{e: minimizers} when $\g >0$. Correspondingly, \eqref{e: minimizers conf invar2} holds only when the deficit is not too large, since we have chosen to normalize the right-hand side by the geometric quantity $\vol_g(M)^{ 1/2^*}$. We thank Rupert Frank for bringing this to our attention.
	}
\end{remark}
 Notice that in Theorem~\ref{thm: minimizers} and Corollary~\ref{cor: conformal CONQUEST}, we obtain a quadratic stability estimate only for a generic set of metrics. This result is in fact sharp. Indeed, adapting an example of Schoen \cite{SchoenNumber} (see also \cite{CaChRu}), we show that there exist manifolds for which $\gamma > 0$ in \eqref{e: minimizers}, thus proving the optimality of the result.

\begin{theorem}[Super Quadratic Growth]\label{thm: superquadratic}
Let $n \geq 3$. There exist  $\gamma > 1$, a closed Riemannian manifold with analytic metric, $(M^n,g)$, a unique minimizer of the Yamabe energy $Q$ on $(M^n, g)$, which we set equal to $1$ (by a conformal change), and a sequence of $u_i \in W^{1,2}(M)$ with $u_i \rightarrow 1$ in $W^{1,2}$ such that 
\begin{equation}\label{e: superquadratic}
\lim_{i\rightarrow \infty} \frac{Q(u_i) - Y(M,[g])}{\|u_i - 1\|^{2+\gamma}_{W^{1,2}(M)}} = 0.
\end{equation}
\end{theorem}

In fact, adapting the aforementioned examples from \cite{SchoenNumber, CaChRu}, we will produce an example such that \eqref{e: superquadratic} holds for any $\gamma < 2$. It is an interesting question whether for every $\gamma>0$ one can find a metric $g_\gamma$ which satisfies \eqref{e: superquadratic}, as in the case of the quantitative isoperimetric inequality on a Riemannian manifold, \cite{ChEnSp}. 

\subsection{Background on Quantitative Stability and the Yamabe Functional}

The problem of establishing quantitative stability estimates for functional and geometric inequalities has been a topic of extensive study in recent years. For instance, sharp quantitative estimates have been established for the isoperimetric inequality on Euclidean space \cite{FMP, FiMaPr, CiLe}, the round sphere \cite{BDF}, hyperbolic space \cite{BDS}, and on arbitrary Riemannian manifolds \cite{ChEnSp}. Closely related to the Yamabe problem, quantitative stability estimates for Sobolev inequalities on Euclidean space have been studied, in addition to the aforementioned result of \cite{BianchiEgnell91}, in \cite{CFMP, FiMPsob, FN, Neumayer, FiZhang, hynd2019symmetry}. In a slightly different direction, quantitative stability estimates for critical points have been addressed for the isoperimetric inequality on Euclidean space \cite{CM, KM} and for the Sobolev inequality \cite{CFM, FGcps}.  Quantitative stability estimates have wide-ranging applications to contexts including characterization of minimizers in variational problems \cite{CicaleseSparado13}, rates of convergence of PDE \cite{CarlenFigalli}, regularity of interfaces in free boundary problems \cite{AKN1}, and even data science \cite{GTMT}.  Apart from \cite{ChEnSp}, all of these results make crucial use of the explicit form of minimizers and critical points or of the symmetries of the ambient space. See \cite{FuSurvey} for a survey of quantitative stability results for functional and geometric inequalities.

Critical points of the volume-normalized Einstein-Hilbert action functional, $\mathcal{R}(g) = \vol_g(M)^{-2/2^*}\int_{M} R_g\,d\vol_g$,
are Einstein metrics, i.e. metrics $g$ satisfying $\Ric_g = \lambda g$ for some $\lambda \in \R$ where $\Ric_g$ is the Ricci curvature tensor of $g$. The Yamabe functional $Q$ defined in \eqref{eqn: Yamabe functional} is the restriction of this functional (and thus the corresponding variational problem) to a given conformal class $[g]$.
If $Y(M, [g]) \leq 0$, then the Euler-Lagrange equation corresponding to the Yamabe functional \eqref{eqn: Yamabe functional} (see \eqref{main} below) satisfies the maximum principle and thus there is a unique critical Yamabe metric. Similarly, if the conformal class $[g]$ has a representative that is an Einstein metric and is not conformal to the round sphere, then this metric is the unique critical Yamabe metric thanks to a theorem of Obata \cite{Obata}. On the round sphere $(S^n, g_0)$, the family of minimizing Yamabe metrics is noncompact, though for any closed Riemannian manifold that is not conformal to the round sphere, the family of unit-volume minimizers is compact in the $C^2$ topology (see Lemma~\ref{lem: cpt}). In fact, Anderson \cite{And} showed that for an open dense set in the space of conformal classes, $[g]$ has a unique (unit volume) minimizing Yamabe metric. 
In general, however, minimizing Yamabe metrics are non-unique; see \cite{SchoenNumber, SchYamSpheres}. 
 Pollack \cite{Pollack1, PollackThesis} showed that, for any $N\in \mathbb{N}$, the set of conformal classes containing at least $N$ critical Yamabe metrics is dense in the $C^0$ norm on the space of conformal classes with positive Yamabe constant. Suitably normalized families of critical points of $Q$ are  compact in the $C^2$ topology for $n\leq 24$ \cite{SchoenNumber, LiZhu, Druet, KMS}, while compactness may fail for $n \geq 25$  \cite{BrendleNonCpt, BrendleMarques} or when the metrics are non-smooth \cite{BerMal}.
 Further related areas of study include the Yamabe problem on compact manifolds with boundary \cite{Escobar1, Escobar2} and the Yamabe flow \cite{BrendleFlow1, BrendleFlow2, CaChRu}.
For further literature review on the Yamabe problem, we refer the reader to \cite{DruetHebey, BrendleMarquesSurvey}.

\subsection{Description of the proof} The proof of Theorem \ref{thm: minimizers} makes use of the so-called \L ojasiewicz inequality, while the generic statement follows from the fact that for generic conformal classes of metric on a given manifold critical points of the Yamabe functional are non-degenerate. By non-degenerate, we mean that the second variation of the Yamabe functional has trivial kernel.  The (distance) \L ojasiewicz inequality originates in real analytic geometry, and roughly says that a real analytic function $q: \R^k \to \R$ grows at least like a power of the distance to the nearest critical point (or to a given level set) of $q$; see Lemma~\ref{lem: lojdistance} below. We apply the \L ojasiewicz inequality to the restriction of the Yamabe energy $Q$ to the kernel of the second variation. In doing so are able to show that in a neighborhood of $v$, the Yamabe energy grows away from $\mathcal{M}$ at least like a power of the distance $d(u, \mathcal{M})$. The connection between the \L ojasiewicz inequality and quantitative stability inequalities was first introduced in \cite{ChEnSp} for the isoperimetric problem.

We remark that, in contrast to \cite{ChEnSp}, our main theorems do not require the analyticity of the metric. This distinction arises from the difference between the area functional considered in \cite{ChEnSp} and the Yamabe functional considered here, namely, that on any closed Riemannian manifold $(M,g)$, the Yamabe functional is an analytic map with respect to $u\in W^{1,2}(M)$ in the sense of \cite[Definition 8.8]{ZeidlerBook}; see \cite[Lemma 6]{CaChRu}. This analyticity allows us to apply the \L ojasiewicz inequality.  While the ``gradient-\L ojasiewicz" inequality has been used before to study Yamabe flows (c.f. \cite{CaChRu, BrendleFlow1}), our paper is the first use of the ``distance-\L ojasiewicz" inequality in the Yamabe literature of which we are aware.

The proof of Theorem \ref{thm: superquadratic} exploits ideas of Adam-Simon \cite{AdSi}, where the notion of Adam-Simon condition of order $p$ was introduced, together with the examples constructed in \cite{CaChRu}.

\subsection{Acknowledgments}
  M.\ Engelstein was partially supported by a NSF DMS 2000288. R.\ Neumayer was partially supported by NSF DMS 1901427 and NSF RTG 1502632. L.\ Spolaor was partially supported by NSF DMS 1951070. 
This project was begun while M.\ Engelstein was visiting Chicago for the AY 2019-2020. He thanks the University of Chicago and especially Carlos Kenig for their hospitality. We thank Otis Chodosh, Dan Pollack, and Rupert Frank for useful conversations about this work. We also thank an anonymous referee for their helpful comments on an earlier draft of this manuscript.

\section{Properties of the Yamabe Energy and Lyapunov-Schmidt reduction}\label{s: prelims}

Throughout, we fix a background metric $g \in [g]$. This conformal representative $g$ is implicit in the definition of the Sobolev function spaces. However, as we saw in Section \ref{s:intro}, our end results in Corollary~\ref{cor: conformal CONQUEST} will be independent of the choice of conformal representative.

Recall the Yamabe energy: 

$$Q(u) =\ \frac{ \int_M c_n |\na u|^2 + R_gu^2 \, d\vol_g}{\| u\|_{L^{2^*}(M)}^2}.$$

 A non-negative critical point $u$ of $Q$ is a non-negative smooth solution of the nonlinear eigenvalue problem
\begin{equation}\label{main}
-c_n \Delta u+ R u = \lambda u^{2^*-1},
\end{equation}
where the value of $\lambda$ is given by $\lambda = Q(u) \| u\|_{L^{2^*}(M)}^{2-2^*}$. We will denote by $\mathcal{CSC}([g]) \subset W^{1,2}(M)$ the set of all critical points in a given conformal class $[g]$, i.e. solutions to \eqref{main} for some $\lambda \in \mathbb R$. As usual, we will omit the dependence on the conformal class when clear from the context.

Although $Q(cu ) = Q(u)$ for any $c >0$,  it will often be easier to work with functions that have $L^{2^*}$ norm equal to  $1$. To that end we introduce the following Banach manifold:

\begin{equation}\label{e:defofB}
 \B = \left\{u\in W^{1,2}(M ; \R_+) \mid \int_M u^{2^*}\ d\vol_g = 1\right\}.\end{equation}

{ Note that the collection of metrics represented by \eqref{e:defofB} is conformally invariant; this can be seen in the equivalent condition that the metric $g_u = u^{4/(n-2)}g$ has unit volume.} 

\begin{lemma}[Banach manifold of metrics of volume $1$]\label{lem:bm}
	The set $\B\subset W^{1,2}(M)$ is a Banach manifold, and for every $v\in \B$ the tangent space to $\B$ is given by
	$$
	T_v \B = \left\{u \in W^{1,2}(M) \mid  \int_{M} v^{2^*-1} u \, d \vol_g =0\right\}\,.
	$$
We will denote by $\pi_{T_v\B}$ the $L^{2}$-orthogonal projection onto $T_v\B$. In particular, for every $u\in \B$ the second variation of $Q$ on $\B$ is given by
	\begin{equation}\label{e:secondvariation1}
	\begin{aligned}
	\frac{1}{2}\nabla_\B^2 Q(u)[\varphi,\eta]
	 =& \int_M \left\{ c_n \,\na \pi_{T_u\B}\varphi\,\cdot \na \pi_{T_u\B}\eta + R_g\,(\pi_{T_u\B}\varphi) \,(\pi_{T_u\B}\eta) \right\}\, d\vol_g\\ -&(2^*-1)Q(u)\int_{M} u^{2^*-2}\, (\pi_{T_u\B}\varphi) \,(\pi_{T_u\B}\eta)\, d\vol_g.\end{aligned}
	\end{equation}
 for all $ \varphi,\eta\in W^{1,2}(M)$.	We will often omit the projection maps when we are doing computations with $\nabla_{\B}^2Q$.

In the special case that $g$ is a metric of constant curvature with volume 1 and $u = 1$ we have the formula (omitting the projection maps):

\begin{equation}\label{e:secondvariation2}
\frac{1}{2}\nabla_\B^2Q(1)[\varphi, \eta] =\frac{4}{n-2}\int -(n-1)(\Delta \varphi)\eta -R_g\varphi \eta\, d\vol_g.
\end{equation}

 Moreover, the following properties hold.
	\begin{enumerate}
		\item  The function $w \mapsto \frac{\nabla_{\B}^2 Q(w)[\eta, -]}{\|\eta\|_{C^{2,\alpha}}}$ is a continuous function from $C^{2,\alpha}\cap \B\rightarrow C^{0,\alpha}$ with a modulus of continuity uniform over $\eta \in C^{2,\alpha}$. 
		\item The function $w\mapsto  \frac{\nabla^2_{\B}Q(w)[\eta, \xi]}{\|\eta\|_{W^{1,2}} \|\xi\|_{W^{1,2}}}$ is a continuous function from $\B \rightarrow \mathbb R$ with modulus of continuity uniform over $\xi, \eta \in W^{1,2}$. 
	\end{enumerate}
\end{lemma}

\begin{proof} 
Since $W^{1,2}(M)$ is separable, to check that $\B$ is a Banach submanifold of $W^{1,2}(M)$, it suffices to check that the function $G\colon W^{1,2}(M)\to \R$ defined by 
$G(u):= \int_M u^{2^*}\ d\vol_g - 1$ 
is a submersion in a neighborhood of every point $v\in W^{1,2}(M)$. 

This is an easy exercise, since we have
	$$
	DG(v)[\varphi]:=\int_M v^{2^*-1}\varphi\,d\vol_g \qquad \forall \varphi\in W^{1,2}(M)\,,
	$$ 
	so that choosing $\varphi=v$ (or $\varphi=1$ since $v>0$ anyway) we get $DG(v)[\varphi]=1\neq 0$. In the sequel, given $v \in\B$, we will denote by $L_v$ the linear (continuous) operator on $W^{1,2}$ defined by $L_v:=DG(v)$ and by $T_v\B$ the tangent space to $\B$ at $v$, which is a codimension $1$ subspace of $W^{1,2}$ defined by
	$$
	T_v \B = \left\{u \in W^{1,2}(M)\mid L_v u =0\right\}.
	$$ 
	Define 

	the orthogonal projection $\pi_{T_v\B} : W^{1,2}(M) \to T_v\B\subset W^{1,2}(M)$  by
	\[
	\pi_{T_v\B} u = u - \left(\int v^{2^*-1} u\right)v\,.
	\]
	Let 
	 us denote
	$$
	\mathcal E(u):=\int_M c_n |\nabla u|^2+R_gu^2\,d\vol_g\,,
	$$ 
	and observe that if $u\in \B$, then $\mathcal E(u)=Q(u)$. For $u\in \B$ and $\varphi \in W^{1,2}(M)$,  we can compute the first variation of $Q$ at points of $\B$ to be:
	\begin{equation}\label{e:firstvariation}\begin{aligned}
	\nabla Q(u)[\phi] &:= \frac{d}{dt}(Q (u + t\phi))\Big|_{t=0}\\ 
	&=\left( [\vol(M, g_{u+t\phi})]^{-2/2^*}\, 2\left[\int_M(c_n( \nabla u\cdot\nabla \phi+\frac{t}{2}\,|\nabla \phi|^2)+R_g(u\,\phi+\frac{t}{2}\phi^2))\,d\vol_g\right]\right.\\
	 &\left.  -2\,\left[\vol(M, g_{u+t\phi}) \right]^{-(2+2^*)/2^*}\, \left[\int_M (u+t\phi)^{2^*-1} \,\phi \,d\vol_g\right]  \,\mathcal E(u+t\phi)\right)\Big|_{t=0}
	\\
	&= 2 \int_{M} \left(-c_n\Delta u + R_gu-Q(u)u^{2^*-1}\right)\phi\, d\vol_g\,,
	\end{aligned}\end{equation} 
	so that in particular, when restricted to the tangent space of $\B$, we have
	$$
	\nabla_{\B} Q(u)[\varphi] = 2\int_{M} \left(-c_n\Delta u + R_gu\right)\pi_{T_u\B}\varphi\, d\vol_g\,.
	$$

	Differentiating \eqref{e:firstvariation} we obtain
	\begin{equation}\label{e:secondvariation}\begin{aligned}
	\nabla^2 Q(u)[\phi,\phi] &:= \frac{d^2}{dt^2}\left(Q (u + t\phi)\right)\Big|_{t=0}\\ 
	& = 2 \int_M \left(c_n\,|\nabla \phi|^2+R_g\,\phi^2\right)\,d\vol_g-2\,(2^*-1)\,Q(u)\,\,\int_M u^{2^*-2}\,\phi^2\,d\vol_g\\
	&+\left(\int_M u^{2^*-1}\,\phi\,d\vol_g\right)\cdot \mathcal G(u,\phi)\,,
	\end{aligned} \end{equation}
	for some smooth function $\mathcal G$. Restricting to $T_u\B$ so that $\int_{M}u^{2^*-1}\, \phi\, d\vol_g = 0$, we exactly obtain \eqref{e:secondvariation1}. After observing that $Q(1) = R_g$ (which is constant) when $g$ is a metric of constant curvature and volume equal to $1$, \eqref{e:secondvariation2} follows from some arithmetic.

	To conclude the proof, let $L_u \vphi = -c_n \Delta \vphi + R_g \vphi - (2^*-1)Q(u)u^{2^*-2}\vphi$. Then we can see that $$\|L_u\vphi - L_v\vphi\|_X \leq C\left(\|u^{2^*-2}\vphi\|_X |Q(u)- Q(v)| + \|\vphi |u^{2^*-2}-v^{2^*-2}|\|_X\right),$$ where $X$ is either the $C^{0,\alpha}$ or $H^{-1}$ norm. If $X = C^{0,\alpha}$, then we recall the continuity of $Q(-)$ in $C^{2,\alpha}$ and note that $x \mapsto x^{2^*-2}$ is continuous to get that $$\|L_u\vphi - L_v\vphi\|_{C^{0,\alpha}} \leq \omega(\|u-v\|_{C^{2,\alpha}})\|\vphi\|_{C^{2,\alpha}},$$ for some modulus of continuity $\omega$.
	
	Similarly if $X = H^{-1}$ we observe that $Q(-)$ is continuous with respect to $u\in W^{1,2}$. Furthermore $\|u^{2^*-2}\vphi\|_{H^{-1}} \leq \|u^{2^*-2}\|_{L^{n/2}}\|\vphi\|_{L^{2^*}}\leq C(\|u\|_{L^{2^*}}) \|\vphi\|_{W^{1,2}}$ and similarly 
	$$\||u^{2^*-2}-v^{2^*-2}|\vphi\|_{H^{-1}} \leq \|u^{2^*-2}-v^{2^*-2}\|_{L^{n/2}}\|\vphi\|_{L^{2^*}}\leq \omega(\|u-v\|_{W^{1,2}}) \|\vphi\|_{W^{1,2}},$$ for some modulus of continuity $\omega$. 
	
Thus to finish the proof of the result, it suffices to show that the map $w \mapsto \pi_{T_w\B}$ is a continuous function from $C^{2,\alpha}\cap \B\rightarrow \mathbf B(C^{2,\alpha}, C^{2,\alpha})$ (or that it is a continuous function from $\B \rightarrow \mathbf B(W^{1,2}, W^{1,2})$). \\
	The triangle inequality shows that $$\left|\int_{M}u^{2^*-1}\eta\, d\vol_g-\int_{M}w^{2^*-1}\eta\, d\vol_g \right| \leq C\|\eta\|_{C^{2,\alpha}}\|w-u\|_{C^{0,\alpha}(M)}.$$ Thus the projection has the desired continuity in the H\"older setting. 
	
	Similarly, H\"older's inequality and the Sobolev embedding $W^{1,2} \hookrightarrow L^{2^*}$ imply that $u \mapsto L_u$ is a continuous function from $W^{1,2} \rightarrow (W^{1,2})^*$, which implies that the projection has the desired continuity in the Sobolev setting. 
\end{proof}

It will be useful to have two additional definitions. First, given a function $v \in \B$, we let $\B(v,\delta)$ denote the $W^{1,2}(M)$ ball of radius $\delta$ centered at $v$ inside of $\B$, i.e.
\begin{equation}\label{notation W12 nbhd, no volume constraint}
	\B(v,\delta) = \{ u \in \B \ \mid \ \| u-v\|_{W^{1,2}(M)}\leq \delta\}.
\end{equation}

Second, we let $\M_1 := \M \cap \B$ and $\mathcal{CSC}_1:=\mathcal{CSC}\cap \B$, that is respectively the minimizers and critical points to the Yamabe functional with $2^*$-norm equal to one.

\subsection{Lyapunov-Schmidt Reduction} 
The following technical result will be key to proving Proposition~\ref{lem: Fuglede}. 
Briefly, Lemma \ref{lem: LS reduction}, called a Lyapunov-Schmidt reduction, see e.g. \cite{Simon0}, splits any perturbation of a critical point into a portion that quantitatively changes the energy to second order and a portion that lies inside of a finite dimensional subspace (which can be dealt with using the \L ojasiewicz inequalities \cite{Loj}).

Given $v \in \mathcal{M}_1$, we let $K= \ker \nabla^2_{\B} Q(v)[-,-]\subset T_v\B$, thinking of the latter as a an operator from $T_v\B\subset W^{1,2}(M) \rightarrow H^{-1}(M)$. Since $\nabla^2_{\B}$ is generated by an elliptic operator on a compact manifold we know $\dim K := l < \infty$. We let $K^\perp$ denote the orthogonal complement of $K$ in $W^{1,2}(M)$ with respect to the $L^2$ inner product. 

\begin{lemma}[Lyapunov-Schmidt Reduction]\label{lem: LS reduction}
	Let $(M,g)$ be a closed Riemannian manifold with $g\in C^3$ and fix $v \in \mathcal{\M}_1$. There is a open neighborhood $U\subset K$ of $0$ in $K$
	and a map 
	\[
	F: U \to K^{\perp} 
	\]
	 with $F(0)=0$ and $\nabla F(0)=0$ satisfying the following properties.
\begin{enumerate}
	\item\label{item: LS property 1} Let $q: U\to \R$ be the function defined by $q(\vphi)= Q(v+ \vphi + F(\vphi))$.  Then we have 
	\begin{equation}\label{eqn: LS proj sat volume constraint}
	\mathcal L := \{v+\vphi + F(\vphi)\mid \vphi \in U\} \subset \B
	\end{equation} and 
	\begin{equation}\label{eqn: finite dim}
	\begin{split}
		\nabla_\B Q(v+ \vphi + F(\vphi)) &= \pi_{K} \nabla_\B Q(v+ \vphi + F(\vphi)) \\
		&= \na q(\vphi).	
	\end{split}	
	\end{equation}
	Furthermore, $\vphi \mapsto q(\vphi)$ is real analytic.

\item\label{item LS 2 W12 nbhd} There exists $\delta>0$  depending on $v$ such that for any $u\in \B(v,\delta)$, we have $\pi_K (u-v) \in U$. Furthermore, if $ u \in \mathcal{CSC}_1 \cap \B(v,\delta)$, then 
\begin{equation}\label{eqn: cp in L 1}
 u = v+\pi_K(u-v) + F(\pi_K(u-v)).	
\end{equation}

\item\label{item: LS F estimate}  There exists $C$ such that for all $ \vphi \in  U$ and $\eta \in K$, we have
		\begin{equation}\label{eqn: estimates for DF}
	\begin{split}
\| \nabla F(\vphi)[\eta] \|_{C^{2,\alpha}}  &\leq C \|\eta\|_{C^{0,\alpha}}\,.
	\end{split}
	\end{equation}
\end{enumerate}
\end{lemma} 

Lyapunov-Schmidt reductions have been already performed for the Yamabe functional in a variety of contexts (see, e.g. \cite[Proposition 7]{CaChRu}). However, since our audience may be less familiar with the construction (which is a consequence of the 
inverse function theorem), we include the proof in Appendix \ref{s:proofofLS}.

Associated to the Lyapunov-Schmidt reduction is the notion of integrability (see for instance \cite{AdSi,CaChRu}), which roughly states that all the elements in the kernel correspond to one-parameter families of critical points.

\begin{definition}[Integrability]\label{def: int} {\rm
	A function $v \in \mathcal{CSC}_1$ is said to be {\it integrable} if for all $\vphi \in \ker \nabla^2_\B Q(v)$ there exists a one-parameter family of functions $(v_t)_{t\in (-\delta,\delta)}$, with $v_0 = v$, $\frac{\partial}{\partial t}\big|_{t=0}v_t = \vphi $ and $v_t \in \mathcal{CSC}_1$ for all $t$ sufficiently small.
	}
\end{definition}

\begin{lemma}[$Q$ in the integrable setting]\label{lem: Integrable}
	Let $(M, g)$ be a closed Riemannian manifold and let $v \in \mathcal{M}$. Then $v$ is integrable if and only if $q$ is constant in a neighborhood of $0 \in K$. In particular, if $v \in \mathcal{M}_1$ is an integrable minimizer, then 
	\begin{equation}\label{eqn: integrable csc and L}
	 \mathcal{M}_1\cap \B(v,\delta) =  \mathcal{L}\,,
	\end{equation}
	where $\mathcal L$ is as in Lemma \ref{lem: LS reduction}, Condition \ref{item: LS property 1}.
\end{lemma}
\begin{proof}
	Suppose that $v$ is integrable. We claim that $q$ is constant in a neighborhood of $0 \in K$. We abuse notation and let $\vphi$ refer to a point in $K \cong \mathbb R^\ell$. Suppose to the contrary that $q$ is non-constant. Considering a Taylor expansion of this analytic function,  we express $q$ as 
	\[
	q(\vphi) =q(0) + q_{k_0} (\vphi ) + q_R(\vphi),
	\]
	where $q_{k_0}$ is a degree $k_0$ homogeneous polynomial, the first non-vanishing term in the Taylor expansion, and $q_R$ is the sum of homogenous polynomials of degree $k >k_0.$ Since $q_{k_0}$ is non-constant, we may find some $\vphi \in K$ such that 
	\begin{equation}\label{eqn: nonzero deriv}
		\na q_{k_0}(\vphi) \neq 0 .
	\end{equation}
	For this choice of $\vphi$, we let $u_s =v + \psi_s \in \B$ be the one parameter family of critical points generated by $\vphi$, whose existence is guaranteed by the integrability of $v$, satisfying $\psi_s = 0$, $\frac{d}{ds}|_{s=0}\psi_s =\vphi$, and 
	\begin{equation}\label{eqn: crit pt integrable pf}
	\nabla_\B Q(v+ \psi_s) =0.
	\end{equation}
By \eqref{eqn: cp in L 1}, all critical points of $Q$ in a $W^{1,2}$ neighborhood of $v$ are contained in $\mathcal{L}$, and so for each $s$ we may express $\psi_s$ as 
	\begin{equation*}
		\psi_s = \vphi_s + F(\vphi_s)
	\end{equation*}
	where $\vphi_s \in K$ and $\frac{\vphi_s }{s} \to \vphi$ as $s\to 0.$ (This latter fact follows because $\psi_s/s \to \vphi$ as $s\to 0$ and $\nabla F(0)=0$ by Lemma~\ref{lem: LS reduction}). Note that by \eqref{eqn: finite dim} of Lemma~\ref{lem: LS reduction} and \eqref{eqn: crit pt integrable pf} we have $\na q(\vphi_s) = 0.$ So, we have 
	\begin{align*}
		0 & = \nabla_\B Q(v+ \psi_s ) = \na q(\vphi_s) \\
		& = \na q_{k_0} (\vphi_s) + \na q_R(\vphi_s)\\
		& = |\vphi_s|^{k_0-1} \na q_{k_0} \left(\frac{\vphi_s}{|\vphi_s|}\right) + o(|\vphi_s|^{k_0-1}).
	\end{align*}
	Dividing through by $|\vphi_s|^{k_0-1}$ and letting $s$ tend to zero, we reach a contradiction to \eqref{eqn: nonzero deriv} and conclude that $q$ is constant. 
	
	Now we establish the opposite implication. Suppose that $q\equiv q(0)$ in a neighborhood of $0$, and thus $\na q \equiv 0$ in a neighborhood of $0.$ Choose any $\vphi \in K$. We claim that $\vphi$ generates a one parameter family of critical points, which will show that $v$ is integrable. Indeed, set 
	\begin{equation*}
		\psi_s = s\vphi + F(s\vphi).
	\end{equation*}
	We can see directly from Lemma~\ref{lem: LS reduction} that 
	\begin{equation*}
		\nabla_{\B} Q(v+\psi_s) = \na q(s\vphi) = 0,
	\end{equation*}
	and so $\psi_s$ is a family of critical points and $v$ is integrable.

	Finally, we show \eqref{eqn: integrable csc and L}. One containment in \eqref{eqn: integrable csc and L} is precisely \eqref{eqn: cp in L 1}, as $\mathcal{M}_1 \subset \mathcal{CSC}_1$. The opposite containment holds in the case that $v$ is integrable, as we have just shown that $q$ is constant on all of $\mathcal{L}$ , and thus all these points are minimizers as well. 
\end{proof}

\section{Local Quantitative Stability of Minimizers}\label{s:localstability}
In this section, we establish the local version of Theorem~\ref{thm: minimizers}, that is Proposition \ref{lem: Fuglede} below. For this we need a localized measure of how far $u$ is from being a minimizer {\it that is close to some given minimizer $v$}.

Given $\delta >0$ and $v \in \mathcal{M}_1$, 
we let 
\[
d_\delta(u , \mathcal{M}_1) = \frac{\inf\left\{ \| u -\tilde{v}\|_{W^{1,2}(M)} \mid \tilde{v} \in \mathcal{M}_1\cap \B(v,\delta)\right\}}{\|u\|_{W^{1,2}(M)}}.
\]

\begin{proposition}[Local Stability Estimate]\label{lem: Fuglede}
	Let $(M,g)$ be a closed Riemannian manifold, and let $v \in \mathcal{M}_1$. Then there exist constants $c, \g$ and $\delta$ depending on $v$ such that 
	\begin{equation}\label{e: Fuglede}
		Q(u) - Y(M) \geq c \,d_\delta(u, \mathcal{M}_1)^{2+\g}\qquad  \text{for all } u \in \B(v,\delta).
			\end{equation}
	If  $v$ is integrable or non-degenerate, then we may take $\g=0$.  
\end{proposition}

We recall that $v$ is called {\it non-degenerate} if the kernel $K = \ker \nabla^2_{\B} Q(v)[-,-]\subset T_v\B$ of the second variation (as in the discussion preceding Lemma~\ref{lem: LS reduction}) is trivial. We also recall that the definition of integrability was given in Definition~\ref{def: int}.

\begin{proof}[Proof of Proposition \ref{lem: Fuglede}]
Given $v \in \M_1$, let $F$ be the Lyapunov-Schmidt reduction adapted to $v$ as in Lemma \ref{lem: LS reduction}, and let $K$ be the kernel of $\nabla_{\B}^2Q(v)$ (see the discussion before Lemma \ref{lem: LS reduction}). 

By Lemma~\ref{lem: LS reduction}\eqref{item LS 2 W12 nbhd}, for any $u\in \B(v,\delta)$, we may define the Lyanpunov-Schmidt ``projection'' $u_{\mathcal{L}}$ of $u$ by
\begin{equation}\label{eqn: LS Projection}
	u_\mathcal{L}  = v + \pi_K (u-v) + F(\pi_K (u-v)).
\end{equation}
	Note that, thanks to  Lemma~\ref{lem: LS reduction}~\eqref{item LS 2 W12 nbhd} and \eqref{item: LS F estimate},  for any $\e>0$, we may take $\delta>0$ small enough in Lemma~\ref{lem: LS reduction} such that 
\begin{align}\label{eqn: uL-v small}
	\| u_\mathcal{L} - v\|_{W^{1,2}(M)} & \leq \e, \\
	\label{eqn: u-uL small}
	\| u_{\mathcal{L}} - u\|_{W^{1,2}(M)} & \leq \e. 
\end{align}

We can write \begin{equation}\label{e:qsplits}
Q(u) -Y = \underbrace{Q(u)-Q(u_{\mathcal L})}_{I} + \underbrace{Q(u_{\mathcal L})-Y}_{II}
\end{equation}
and estimate these two terms separately.

\noindent {\bf Term $I$:} It will be useful for us to write $u = u_{\mathcal L} + u^{\perp}$. Using the notation introduction before Lemma \ref{lem: LS reduction} we note that $u^\perp \in K^\perp$. 
To estimate $I$, we use Taylor's theorem and see that 
\begin{equation}\label{e:first taylors theorem} Q(u)-Q(u_{\mathcal L}) = \nabla_\B Q(u_{\mathcal L})[u^\perp] + \frac{1}{2}\nabla^2_{\B}Q(\zeta)[u^\perp, u^\perp], 
\end{equation} 
for some $\zeta$  on a geodesic in $\B$ between $u$ and $u_{\mathcal L}$. Observe that $\nabla_\B Q(u_{\mathcal L})[u^\perp]=0$ by Lemma \ref{lem: LS reduction} and the fact that $u^\perp \in K^\perp$. 
Furthermore, using the continuity of $\nabla^2_{\B} Q(-)$ established in Lemma \ref{lem:bm} and \eqref{eqn: u-uL small}, we can write
 \begin{equation}\label{e:movethesecondderivative}
 Q(u)-Q(u_{\mathcal L}) = \frac{1}{2}\nabla^2_{\B} Q(v)[u^\perp, u^\perp] + o(1)\|u^{\perp}\|_{W^{1,2}}^2,
 \end{equation}
  where $o(1)$ represents a term that goes to zero as $\|u-v\|_{W^{1,2}} \rightarrow 0$. Let $\lambda_1 > 0$ be the smallest non-zero eigenvalue of $\nabla^2_\B Q(v)$. It then follows that, picking $\delta > 0$ in the statement small enough, \begin{equation}\label{e:finalcoerciveestimate}Q(u)-Q(u_{\mathcal L}) \geq  \frac{1}{2}\lambda_1 \|u^\perp\|_{W^{1,2}}^2 +o(1)\|u^{\perp}\|_{W^{1,2}}^2 \geq \frac{1}{4}\lambda_1\|u^{\perp}\|_{W^{1,2}}^2.\end{equation} 

\medskip

\noindent {\bf The Term $II$:} It will be useful to separate out three cases for estimating term $II$.

\smallskip

\noindent {\bf $v$ is non-degenerate.} This is the easiest case, since then $u_{\mathcal L}=v$ and then \eqref{e:finalcoerciveestimate} concludes the proof.

\smallskip

\noindent {\bf $v$ is integrable.} By Lemma \ref{lem: Integrable} we have that $Q(u_{\mathcal L}) = q(\pi_K(u-v)) = q(0) = Q(v)$. So the proposition follows from \eqref{e:finalcoerciveestimate}. 

\smallskip

\noindent {\bf $v$ is non-integrable.} Let $\vphi = \pi_K(u-v)$ and recall that $Q(u_{\mathcal L}) = q(\vphi)$. We know that $\vphi \mapsto q(\vphi)$ is an analytic function  $\mathbb R^\ell\rightarrow \mathbb R$ where $\ell = \dim K$. Thus we can apply the \L ojasiewicz inequality \cite{Loj}: 

\begin{lemma}[\L ojasiewicz ``distance" inequality]\label{lem: lojdistance}
Let $q: \mathbb R^\ell \rightarrow \mathbb R$ be a real analytic function and assume that $\nabla q(0) = 0$. Then there exist  $\tilde{\delta} > 0, c > 0$ and $\g > 0$ (all of which depend on $q$ and on the critical point $0$)
such that for all $\vphi \in B(0,\tilde{\delta})$,
 \begin{equation}\label{e: lojdistance}
|q(\vphi) - q(0)| \geq c \inf\left\{ |\vphi - \bar \vphi | \ : \ \bar \vphi \in B(0,\tilde{\delta}), \, \na q(\bar \vphi ) =0\right\}^{2+\g}.
\end{equation}
\end{lemma}

 Appealing to the definition of $q$ in Lemma~\ref{lem: LS reduction} and   the $\L$ojasiewicz inequality in  Lemma~\ref{lem: lojdistance}, we see that 
\begin{equation}\label{eqn: intermediate Loj} 
\begin{split}
	Q(u_{\mathcal{L}}) -Y& =
	q(\vphi) - q(0) \\
	& \geq c \inf\{ |\vphi - \hat{ \vphi} | \ : \  \hat{\vphi} \in K\cap B(0,\delta), \na q(\vphi ) =0\}^{2+\g}.
\end{split}	
\end{equation}
Notice further that 
\begin{align*}
 \inf\{ |\vphi - \bar \vphi | \ :&\ \bar \vphi \in K\cap B(0,\delta), \na q(\vphi ) =0\} \\
 &\geq c \inf \{ \|u_\mathcal{L} - \hat{v} \|_{W^{1,2}(M)} \ : \ \hat{v} \in \M_1 \cap \B(v,\delta)\} 
 \end{align*}
because for any $\hat{v} \in \M \cap \B(v,\delta)$, we may write $\hat v = v + \hat{\vphi} + F(\hat{\vphi})$ for some  $\hat{\vphi} \in K\cap B(0,\delta)$ with  $\na q(\vphi ) =0$, and 
\begin{align*}
\|u_\mathcal{L} - \hat{v}\|_{W^{1,2}(M)} &= \| \pi_K (u-v) +F(\pi_K(u-v)) -  \hat{\vphi} -F(\hat{\vphi})\|_{W^{1,2}(M)} \\
& \leq \| \pi_K (u-v) -\hat{\vphi}\|_{W^{1,2}(M)} + \| F(\pi_K(u-v))-F(\hat{\vphi})\|_{W^{1,2}(M)}\\
& \leq \| \pi_K (u-v) -\hat{\vphi}\|_{W^{1,2}(M)} + C \| \pi_K (u-v) -\hat{\vphi}\|_{C^{0,\alpha}(M)}\\
&\leq C\| \pi_K (u-v) -\hat{\vphi}\|_{W^{1,2}(M)},
\end{align*}
where in the penultimate inequality we have used Lemma~\ref{lem: LS reduction}\eqref{item: LS F estimate}.
Together with \eqref{eqn: intermediate Loj}, this implies that 
\begin{equation}\label{eqn: final Loj}
	Q(u_{\mathcal{L}}) -Y  \geq 	c \inf \{ \|u_\mathcal{L} - \hat{v} \|_{W^{1,2}(M)} \ : \ \hat{v} \in \M_1 \cap \B(v,\delta)\}^{2+\g}.
\end{equation}
Combining \eqref{eqn: final Loj} with \eqref{e:finalcoerciveestimate} yields the result in this third and final setting. 
\end{proof}

\section{Proofs of Theorems~\ref{thm: minimizers} and \ref{thm: superquadratic} and Corollary~\ref{cor: conformal CONQUEST}}

In this section we conclude the proofs of the main results, that is, Theorems \ref{thm: minimizers} and \ref{thm: superquadratic} and Corollary~\ref{cor: conformal CONQUEST}.  Theorem \ref{thm: minimizers} will be a consequence of the local quantitative stability in Proposition \ref{lem: Fuglede} and a compactness argument, while Theorem \ref{thm: superquadratic} will follow from an example of \cite{CaChRu}.

\subsection{Proof of Theorem \ref{thm: minimizers}}
In the proof of Theorem~\ref{thm: minimizers},  we will make use of the following compactness result for minimizing sequences, which is proven, for instance, in \cite[Theorem 4.1]{Lions}.
\begin{lemma}\label{lem: cpt}
	Let $(M,g)$ be a smooth Riemannian manifold of dimension $n\geq 3$ and let $(u_i)\subset \B$ be a sequence such that $Q(u_i) \to Y$.  Then, up to a subsequence, $u_i$ converges strongly in $W^{1,2}(M)$ to some $v \in \mathcal{M}_1$. 
\end{lemma}
We now prove Theorem~\ref{thm: minimizers}.
\begin{proof}[Proof of Theorem~\ref{thm: minimizers}]
Since both sides of \eqref{e: minimizers} are zero-homogeneous in $u$ and because  $\inf\{ \| u -v\|_{W^{1,2}(M)} : v\in \M_1\} \geq d(u, \M)$,  we may work in $\B$ without loss of generality.

Given $v \in \M_1,$ let $\delta(v)$,  $\g(v)$, and $c(v)$ be the constants given in Proposition~\ref{lem: Fuglede}.
Since the set $\M_1 = \{  v \in \M : \| v\|_{L^{2^*}(M)}=1\}$ is compact in $W^{1,2}$ by Lemma~\ref{lem: cpt}, we may cover $\M_1$ by balls $\B(v, \delta(v)/2)$ and take a finite subcover $\{\B(v_i, \delta(v_i)/2)\}_{i\in \mathscr I}$. Then we define
\begin{align*}
\delta_0 &= \min_{i\in \mathscr I} \delta(v_i)/2 > 0,\\
\g_0 & = \max_{i\in \mathscr I}  \g(v_i)< \infty,\\
c_0 & = \min_{i\in \mathscr I}  c(v_i) >0.
\end{align*}

Let $u\in \B$ be such that $d(u, \mathcal{M}_1) < \delta_0/4$. 
There exists a $i \in \mathscr I$ such that $\|u -v_i\|_{W^{1,2}} < \delta_i/2$. If $\tilde{v}$ is the closest element of $\M_1$ to $u$ the triangle inequality implies that $\|\tilde{v}-v_i\|_{W^{1,2}} < \delta_i$. Thus we may apply Proposition~\ref{lem: Fuglede} to see that 
$$Q(u) - Y(M,[g]) \geq c(v_i)d_{\delta_i}(u, \mathcal{M}_1)^{2+\g_i} \geq c_0d(u, \mathcal M_1)^{2+ \g_0},
$$ which is the desired result. 

We are left with the case that $d(u, \mathcal M_1) > \delta_0/4$. Note since $\|u\|_{L^{2^*}}=1$ we have $d(u, \mathcal M) > \delta_0/16$ by the triangle inequality. 
Thanks to Lemma~\ref{lem: cpt} and the triangle inequality, 
 there exists a $\varepsilon > 0$ such that $$Q(u)-Y(M,[g]) < \varepsilon \Rightarrow d(u, \mathcal M) < \delta_0/16.$$ Thus, when $d(u, \mathcal M_1) > \delta_0/4$ we have that $Q(u) - Y(M,[g]) > \varepsilon $. Moreover, observe that by definition, $d(u, \mathcal M) \leq  1$. 

 Letting $c = \min\left\{c_0, \varepsilon\right\}$ we have proven the stability estimate \eqref{e: minimizers} for all $u \in \B$. 

\medskip

Finally, we show the generic statement. By work of Schoen \cite{SchoenNumber} (see also Anderson \cite{And}), generically (that is for an open and dense subset of the set of equivalence classes of $C^\infty$ metrics on a given compact manifold $M$ in the $C^2$ topology), there are finitely many critical points of $Q$ and each one is non-degenerate. Therefore the proof follows straightforwardly from the local version of Proposition \ref{lem: Fuglede} in the non-degenerate case, that is with $\gamma=0$.
\end{proof}

\subsection{Proof of Corollary~\ref{cor: conformal CONQUEST}}
Corollary~\ref{cor: conformal CONQUEST} is a direct consequence of Theorem~\ref{thm: minimizers}, up to showing that the distances defined there are conformally invariant.

\begin{proof}[Proof of Corollary~\ref{cor: conformal CONQUEST}]
Let $g = \phi^{4/(n-2)}\hat{g}.$ Note that 
$Q_{\hat{g}}(u) = Q(\phi u ) $
 and $\mathcal{M}_{g} = \{ v \in W^{1,2}(M) : Q_{g}(v) =Y \} = \{ v \in W^{1,2}(M) :  \phi v \in \mathcal{M}_{\hat{g}}\}.$ So, consider the metric $\tilde{g}$ given by $\tilde{g}= u^{4/(n-2)}g = (u\phi)^{4/(n-2)} \hat{g}$. We directly compute that
\[
\inf_{v \in \mathcal{M}_g} \int_M |u-v|^{2^*}d \vol_g = \inf_{v \in \mathcal{M}_g}\int_M |\phi(u-v)|^{2^*}d \vol_{\hat{g}} = \inf_{w \in \mathcal{M}_{\hat{g}}} \int_M |\phi u-w|^{2^*}d \vol_{\hat{g}},
\] 
which proves that $\|\cdot\|$ is conformally invariant. So, applying Theorem~\ref{thm: minimizers} and the Sobolev inequality on $(M,g)$, with $\tilde{g} = u^{\frac{4}{n-2}}g$, we have 
\begin{align*}
\mathcal{R}_{\tilde{g}} - Y = Q_{g}(u) -Y &\geq c \left(\frac{\inf_{v \in \mathcal{M}} \| u-v\|_{W^{1,2}(M)}}{\| u\|_{W^{1,2}(M)}}\right)^{2+\g} \\
& \geq c\left(\frac{\inf_{v \in \mathcal{M} }\| u-v\|_{L^{2^*}(M)}}{\| u\|_{L^{2^*}(M)}}\right)^{2+\g}= c \left(\frac{\inf_{g \in \mathcal{M}} \| \tilde{g}-g \|}{\vol_{\tilde{g}}(M)^{1/2^*}}\right)^{2+\g}. 
\end{align*}
To see the second inequality above, first note that in the case when $d(u, \mathcal M) \leq \delta_0$, the denominators are comparable. On the other hand, when $d(u, \mathcal M) > \delta_0$, we observe that the quantity ${\inf_{v \in \mathcal{M} }\| u-v\|_{L^{2^*}(M)}}/{\| u\|_{L^{2^*}(M)}}$ is bounded above by $1$ and so the inequality follows by choosing $c$ sufficiently small.
This establishes \eqref{e: minimizers conf invar}.
Next, suppose that $Y>0$ and $ g \in \mathcal{M}_1$, and recall that 
\[
\| g_u- g_v\|_* = \left( \int_{M}c_n |\na (u-v)|^2 + R_g (u-v)^2 \,d\vol_g\right)^{1/2}.
\]
Again as a consequence of Theorem~\ref{thm: minimizers} and the assumption that $\mathcal{R}_g - Y$, we obtain
\[
\mathcal{R}_g - Y \geq c\left(\frac{\inf_{\tilde{g} \in \mathcal{M}} \| g-\tilde{g}\|_*}{ \vol_g(M)^{1/2^*}}\right)^{2+\gamma}.
\] 
where we used that $g\in \mathcal{M}, Y > 0$ and, consequently, that $R_g > 0$ is constant. To conclude, it suffices to prove that $\|g_u - g_v\|_*$ does not depend on the choice of $g$.
Suppose that $g,\hat{g} \in \M_1$ with $g = \phi^{4/(n-2)} \hat{g}$. Then if $g_u = u^{4/(n-2)}g = (\phi u)^{4/(n-2)}\hat{g} = \hat{g}_{\phi u}$ and $g_v = v^{4/(n-2)}g = (\phi v)^{4/(n-2)}\hat{g} = \hat{g}_{\phi v}$, we have 
\begin{align*}
	\|g_u - g_v\|_* & = \int_M c_n | \na_g (u - v)|^2  + R_g (u - v)^2 \,d\vol_g\\
	& = \int_M (u-v)\left(-c_n \Delta_g(u-v) + R_g(u-v)\right)\, d\vol_g.
	\end{align*}
	Recall that $-c_n \Delta_g + R_g \equiv L_g$ is the conformal Laplacian and we have \begin{align*} L_g \psi =& \phi^{1-2^*}L_{\hat{g}}(\phi \psi)\\
	L_g \psi\, d\vol_g =& \phi L_{\hat{g}}(\phi \psi)\, d\vol_{\hat{g}}.
	\end{align*}
	Plugging this into the above we get that 
	\begin{align*}
	\|g_u - g_v\|_*& = \int_M (u-v)L_g(u-v)\, d\vol_g \\
	&= \int_M (\phi u - \phi v)L_{\hat{g}}(\phi u - \phi v)\, d\vol_{\hat{g}} = \|\hat{g}_{\phi u} - \hat{g}_{\phi v}\|_*.
	\end{align*}
This concludes the proof of \eqref{e: minimizers conf invar2}.	

\end{proof}

\subsection{Proof of Theorem \ref{thm: superquadratic}} Suppose that $u_0\in \mathcal M$ is nonintegrable.  Let $q: U \to \R$ where $U \subset \ker \na_{\B}^2Q(v) \cong \R^{\ell}$   be the function defined in Lemma \ref{lem: LS reduction}; since $q$ is analytic we can expand it in a power series 
$$
q(x)=q(0)+\sum_{j\geq p} q_j(x)
$$
where each $q_j$ is a degree $j$ homogeneous polynomial and $p$ is chosen so that $q_p(0)\neq 0$. As in \cite{CaChRu}, we will call $p$ \emph{the order of integrability of $u_0$}. Next we recall the notion of Adams-Simon positivity condition:

	\begin{definition}[$\ASp$ condition]\label{def:ASp}
		We say that $u_0$ satisfies the \emph{Adams–Simon positivity condition of order $p$}, $\ASp$ for short, if $p$ is the order of integrability of $u_0$ and $q_p|_{\mathbb S^{\ell -1}}$ attains a positive maximum for some $v\in \mathbb S^{\ell-1}$. 
	\end{definition}

The following Proposition is immediate from the definitions.

\begin{proposition}[$\ASp$ implies $\gamma>0$]\label{p:banana} Fix a closed Riemannainan manifold of dimension $n\geq 3$ and fix $p \geq 3.$ 
	Let $u_0$ be a non integrable critical point of the Yamabe energy and suppose that it satisfies the Adams--Simon positivity condition of order $p$. Then there exists a sequence of $u_i \in W^{1,2}(M)$ with $u_i \rightarrow u_0$ in $W^{1,2}$ but
	\begin{equation}\label{e:banana}
	\lim_{i\rightarrow \infty} \frac{Q(u_i) - Y(M,g)}{\|u_i - u_0\|^{p-\alpha}_{W^{1,2}}} = 0\,,\quad \forall \alpha>0.
	\end{equation}
\end{proposition}

\begin{proof} Let $v\in \mathbb S^{\ell -1}$ be the maximum of $q_p$ as in Definition \ref{def:ASp}. For $t\in [0,1]$ let $\hat u_t:= tv $, and consider the family of functions $(u_t)_t\subset W^{1,2}(\mathcal M)$ defined by
	$$
	u_t:=u_0+\hat u_t +F(\hat u_t)\,\quad t\in (0,1)\,,
	$$
	where $F$ is the function defined in Lemma \ref{lem: LS reduction}. By definition of $u_t$ and the properties of $F$, we have
	$$
	\|u_t-u_0\|_{W^{1,2}}\sim t\,,
	$$
	and moreover, by definition of $q$, we have
	$$
	Q(u_t)-Q(u_0)=q(\hat u_t)-q(0)=\sum_{j\geq p} q_j(\hat u_t)\,.
	$$
	Since $u_0$ satisfies $\ASp$, we conclude
	$$
	|Q(u_t)-Q(u_0)|\leq C t^p q_p(v)
	$$
for $t $ sufficiently small,
	which implies the desired conclusion.
\end{proof}

We are now ready to conclude the proof of Theorem \ref{thm: superquadratic}.

\begin{proof}[Proof of Theorem \ref{thm: superquadratic}] By Proposition \ref{p:banana}, it is enough to prove the existence of compact manifolds $(M,g)$, with $g$ a minimizer of the Yamabe energy, satisfying the $\ASp$ condition for $p\geq 3$. This has been done in \cite{CaChRu} (see also \cite{SchoenNumber}), and we recall them here for completeness.
	\begin{itemize}
		\item[(i)] Fix integers $n, m > 1$ and a closed $m$-dimensional Riemannian manifold $(M^m,g_M)$
		with constant scalar curvature $R_{g_M}=4(n+1)(m+n-1)$. Let $(\mathbb{P}^n,g_{FS})$ be the complex projective space equipped with the Fubini-Study metric, where the normalization of $g_{FS}$ is fixed so that $\mathbb S^{2n+1}(1) \to (\mathbb{P}^n,g_{FS})$ is a Riemannian submersion. Then the product metric  $M^m \times \mathbb P^n , g_M \bigoplus g_{FS} )$ is a degenerate critical point satisfying $\ASp$, $p=3$. 
		\item[(ii)]The product metric on $\mathbb S^1(1/\sqrt{n-2})\times \mathbb S^{n-1}(1)$ is a nonintegrable minimizer of the Yamabe energy satisfying $\ASp$ for some $p\geq 4$ (cf. \cite[Proposition 4]{CaChRu}).
	\end{itemize}
	In particular (ii) provides the desired example, being a minimizer.
\end{proof}

\appendix
\section*{Appendix}
\section{Proof of the Lyaponuv-Schmidt Reduction (Lemma \ref{lem: LS reduction})}\label{s:proofofLS}

	\begin{proof}[Proof of Lemma~\ref{lem: LS reduction}]
Fix $v \in \mathcal{M}_1$ and let $K, K^{\perp}$ be as in the discussion preceding the statement of Lemma~\ref{lem: LS reduction}. We proceed in several steps.
\\

{\it Step 1: Defining the map $F$.} We obtain the map $F$ using the inverse function theorem in the following way.
	Let us consider the map 

  {$\mathcal{N}: C^{2,\alpha}\cap \B \to C^{0,\alpha}(M)\cap T_v \B $}  		 defined by 
	\[
	\mathcal{N}(w) = \pi_K (w-v)+ \pi_{K^\perp}\na_\B Q(w).
	\]

This map satisfies $\mathcal{N}(v) = 0$ and, if $w_t$ is smooth curve in $C^{2,\alpha}\cap \B$ with $w_0 = v$ and $\partial _t|_{t=0}w_t =\zeta$, then
	\begin{equation}\label{eqn: derivative of N}
\begin{split}
	\na_\B \mathcal{N}(v)[\zeta] =\frac{d}{dt}|_{t=0} \, \mathcal{N}(w_t)  &= \pi_K\zeta +\pi_{K^\perp}\na_\B^2 Q(v)[\zeta,-] \\
	& = \pi_K\zeta +\na_\B^2 Q(v)[\zeta,-].
\end{split}	
	\end{equation}
Note that this is well defined for any $\zeta \in C^{2,\alpha}(M)\cap T_v\B$ 
 The last identity follows because $\na^2_\B Q(v)[\zeta] \in K^\perp$ for any $\zeta \in W^{1,2}(M)\cap T_v\B$; indeed, for any $\vphi \in K$ we have $0 = \langle \na^2_\B Q(v)[\vphi],\zeta\rangle_{L^2} = \langle \vphi, \na^2_\B  Q(v)[\zeta]\rangle_{L^2}.$ 

In particular, \eqref{eqn: derivative of N} shows that the kernel of $\na_\B\mathcal{N}(v)$ is trivial, because for any $\zeta\neq 0$, either $\pi_K \zeta \neq 0$ or $\zeta \in K^\perp $, and thus $\na_\B^2 Q(v)[\zeta,-]$ is non-vanishing by definition. 

 Furthermore, because the operator $\zeta \mapsto L\zeta :=\na^2_{\B} Q(v)[\zeta, -]$ is uniformly elliptic, Schauder estimates ensure that $\na_\B\mathcal{N}(v)$ is an isomorphism from $C^{2,\alpha}(M)\cap T_v\B$ to $C^{0,\alpha}(M)\cap T_v\B$. Thus, we may apply the inverse function theorem to obtain an inverse $\mathcal{N}^{-1}$ defined on an open neighborhood $\hat{U} \subset C^{0,\alpha}(M)\cap T_v\B$ containing $0$. Set $U = K\cap \hat{U}\subset K$ and  define the map $ F:  U \to K^{\perp} $
	  by
	  \[
	  F(\vphi) = \pi_{K^\perp} (\mathcal{N}^{-1}(\vphi)-v).
	  \]

 \medskip  
 
 {\it Step 2: Basic observations about the map $F$.}
 Let us make some initial observations that will be useful for proving the claimed properties of $F$. For any $\vphi \in U$, from the definition of $\mathcal{N}$ we have
	  \begin{equation}
	\label{eqn: phi}\begin{split}
		\vphi &=\mathcal{N}(\mathcal{N}^{-1}(\vphi) )\\
		&= \pi_K (\mathcal{N}^{-1}(\vphi)-v)+	\pi_{K^\perp} \na_\B Q(\mathcal{N}^{-1}(\vphi)).
	\end{split}
	\end{equation}
	(Recall that the image of $\mathcal N^{-1}$ is contained in $C^{2,\alpha}(M) \cap \B$ so \eqref{eqn: phi} makes sense). 
	
Taking $\pi_K$ of both sides of \eqref{eqn: phi}, we see that $\vphi = \pi_K (\mathcal{N}^{-1}(\vphi)-v)$. So, along with the definition of $F$, this implies that
	  \begin{equation}\label{eqn: split}
	 \mathcal{N}^{-1}(\vphi)= \vphi + F(\vphi) + v \qquad \text{for all } \vphi \in  U.
	 \end{equation}
	Differentiating \eqref{eqn: split}, we find that for any  $\vphi \in U$ and $\eta \in K$, we have 
\begin{equation}\label{eqn: dF}
\begin{split}
	\pi_{K^\perp} \na \,\mathcal{N}^{-1}(\vphi)[\eta]&=  \nabla \, F(\vphi)[\eta] ,\\
	\pi_{K}\na \, \mathcal{N}^{-1}(\vphi) [\eta]& = \eta.
\end{split}
\end{equation}
Notice that in \eqref{eqn: dF} we can write $\nabla$ instead of $\na_{\B}$, since $\vphi\in C^{2,\alpha}(M)\cap T_v\B$. We will do this several time in what follows.

 \medskip

{\it Step 3: Verifying properties of $F$.}	  
	We now check that this map $F$ satisfies the desired properties in the statement of Lemma \ref{lem: LS reduction}. 
	It is clear that $F(0)=0$ since $\mathcal{N}(0)=0$. To see that $\nabla F(0) = 0$, we appeal to \eqref{eqn: dF} with $\vphi=0$ and see that it suffices to show that $\pi_{K^\perp}\na \mathcal{N}^{-1}(0)[\eta]=0$ for any $\eta\in K$.
And indeed, by \eqref{eqn: derivative of N}, we see that $\na_\B \mathcal{N}(v)$ maps $K$ to $K$ and that $\na_\B \mathcal{N}(v)|_K = (\nabla \mathcal{N}^{-1}(0))^{-1}|_K = \text{Id}.$  	Thus $\na F(0)=0$.

Next, we prove property \eqref{item: LS property 1}.
 First note that $\mathcal N$ is analytic in $w\in \B$ as long as $w\mapsto Q(w)$ is analytic in $w\in \B$

 First note that $\mathcal N$ is analytic in $w\in \B$ in the sense of \cite[Definition 8.8]{ZeidlerBook}  because  $\pi_K,\pi_{K^\perp}$ are linear  and $w\mapsto Q(w)$ is analytic in $w\in \B$; see \cite[Lemma 6]{CaChRu}. It then follows by the inverse function theorem that $F$, and therefore $q$, are analytic functions over $K \cong \mathbb R^\ell$  (see \cite[Theorem 4.H]{ZeidlerBook}). 

To see \eqref{eqn: LS proj sat volume constraint}, recall \eqref{eqn: split}, that $\mathcal N^{-1}(\varphi) = v + \varphi + F(\varphi)$. But we know that the domain of $\mathcal N$ is $\B \cap C^{2,\alpha}(M)$ so it must be that the range of $\mathcal N^{-1}$ is contained in $\B$.

The first equality in \eqref{eqn: finite dim}
follows directly from taking $\pi_{K^\perp}$ of both sides of \eqref{eqn: phi} and recalling \eqref{eqn: split}. To see the second equality in \eqref{eqn: finite dim}, by the chain rule for any $\vphi\in U$ and $\eta \in K$ we have
\begin{align*}
	\frac{d}{dt}q(\vphi + t \eta)|_{t= 0} = \langle \nabla q(\vphi) ,\eta\rangle & = \na_\B Q(v + \vphi + F(\vphi))[\eta + \nabla F (\vphi)[\eta]]\\
	&  = \na_\B Q\left(v + \vphi + F(\vphi)\right)[\eta],
\end{align*}
with the latter term vanishing in the second equality because $\nabla F (\vphi)[\eta]\in K^\perp$ by \eqref{eqn: dF}.

To see property \eqref{item LS 2 W12 nbhd}, note that $U$ contains a $C^{0,\alpha}$ ball of radius $\e$ in $K$ for $\e$ sufficiently small. Since all norms are equivalent in the finite dimensional space $K$, we see that $U$ contains an $L^2$ ball of radius $\e'$ in $K$ for some $\e'$ depending on $\e.$ Now, since the $L^2$ norm is nonincreasing under the $L^2$ projection $\pi_K$, we have 
\[
\| \pi_K (u - v) \|_{L^2(M)} \leq \| u -v\|_{L^2(M)} \leq \| u-v\|_{W^{1,2}(M)}.
\]
So, provided $\delta \leq \e',$ we have that the first claim of property~\eqref{item LS 2 W12 nbhd} holds.
 Next, 
 basic elliptic regularity estimates show that if $u \in \mathcal{CSC}_1 \cap \B(v,\delta)$, we may take $\| u -v\|_{C^{2,\alpha}(M)}$ as small as desired by choosing $\delta$ to be sufficiently small; in particular, for $\delta$ sufficiently small, $u-v$ is contained in the neighborhood in which the map $\mathcal{N}$ is invertible. So, letting $w=u-v$. we have
\begin{align*}
u = \mathcal{N}^{-1}(\mathcal{N} u	)& = \mathcal{N}^{-1}( \pi_K w + \pi_{K^\perp} \na_\B Q(u))\\
& = \mathcal{N}^{-1}(\pi_K w)\\
& = v+\pi_K w + F(\pi_K w),
\end{align*}
where we have used \eqref{eqn: split} in the final equality. This proves \eqref{eqn: cp in L 1}.

Now we show property~\eqref{item: LS F estimate}. To verify the estimate \eqref{eqn: estimates for DF}, we first 
apply Schauder estimates  to find
\begin{equation}\label{eqn: schauder}
	 \|\na F(\vphi)[\eta]\|_{C^{2,\alpha}(M)}
	 	  \leq C \left\| \na^2_\B Q(v)\left[\na F(\vphi)[\eta]	\right]\right\|_{C^{0,\alpha}(M)}.
\end{equation}
From the second identity in \eqref{eqn:  dF}, we  find that 
	 \begin{equation}\label{eqn: a}
	 \begin{split}
 \na^2_\B Q(v)\left[\na\,F(\vphi)[\eta]	\right]
	 	 & = \na^2_\B Q(v)\left[ \pi_{K^\perp} \na\, \mathcal{N}^{-1}(\vphi)[\eta]	\right]\\
&=	  \pi_{K^\perp} \na^2_\B Q(v)\left[\na \mathcal{N}^{-1}(\vphi)[\eta	]\right].
	 \end{split}
	 \end{equation}
	The second equality follows because $\na^2_\B Q(v)$ commutes with $\pi_{K^\perp}$. The reason for this is, as we've seen above, that $\na^2_\B Q(v)[w] \in K^\perp$ for any $w \in W^{1,2}(M)$.

 So, from \eqref{eqn: schauder} and \eqref{eqn: a}, we find that 
\begin{equation}\label{eqn: intermediate LS}
\|\na F(\vphi)[\eta]\|_{C^{2,\alpha}(M)} \leq 
\left\| \pi_{K^\perp} \na^2  Q(v)\left[\na \mathcal{N}^{-1}(\vphi)[\eta]\right]\right\|_{C^{0,\alpha}(M)}
\end{equation}

Next, we claim that 
\begin{equation}\label{eqn: intermediate 2 LS}
	\left\| \pi_{K^\perp} \na^2  Q(v)\left[\na \mathcal{N}^{-1}(\vphi)[\eta]\right]\right\|_{C^{0,\alpha}(M)} \leq \e \| \na \mathcal{N}^{-1}(\vphi)[\eta]\|_{C^{2,\alpha}},
\end{equation}
To this end, we first note that  differentiating \eqref{eqn: phi} in the direction $\eta\in K$, we have 
\begin{align*}
	\eta = \pi_K\na \mathcal{N}^{-1}(\vphi)[\eta] + \pi_{K^\perp} \na^2 Q(\mathcal{N}^{-1}(\vphi))[\na \mathcal{N}^{-1}(\vphi)[\eta]].
\end{align*}
So, by taking $\pi_{K^\perp}$ of both sides, we determine that
\begin{equation}\label{eqn: vanishing term}
	\pi_{K^\perp} \na^2 Q(\mathcal{N}^{-1}(\vphi))[\na \mathcal{N}^{-1}(\vphi) [\eta]]=0.
\end{equation}
So, we can write
\begin{align*}
\big\|
\pi_{K^\perp} &\na^2Q(v)[\na \mathcal{N}^{-1}(\vphi) [\eta]]\big\|_{C^{0,\alpha}}\\
& =\big\| \pi_{K^\perp} \left(\left(\na^2 Q(v)-\na^2 Q(\mathcal{N}^{-1}(\vphi))\right)[\na \mathcal{N}^{-1}(\vphi) [\eta]] \right)\big\|_{C^{0,\alpha}}\\
& \leq \e \| \na \mathcal{N}^{-1}(\vphi)[\eta]\|_{C^{2,\alpha}}.
\end{align*}
The final inequality follows because Lemma \ref{lem:bm}  implies that, for a modulus of continuity, $\omega$ (which may change from line to line):
\begin{equation}\begin{split}
\| \left(\na^2 Q(v)-\na^2 Q(\mathcal{N}^{-1}(\varphi))\right)\|_{C^{2,\alpha}\rightarrow C^{0,\alpha}}
& \leq \omega\left(\|v-\mathcal N^{-1} \vphi\|_{C^{2,\alpha}}\right)\\
& \leq \omega(\|\vphi\|_{C^{0,\alpha}}) \leq \tilde{\omega}(\|\vphi\|_{W^{1,2}}).
\end{split}
\end{equation}
The penultimate inequality follows by the continuity of $\mathcal N^{-1}$ from $C^{0,\alpha} \rightarrow C^{2,\alpha}$. The last inequality follows provided that $\|\vphi\|_{W^{1,2}}$ is sufficiently small (recall that $\vphi \in K$ and all the norms are equivalent on $K$). This establishes \eqref{eqn: intermediate 2 LS}.

Thus far, from \eqref{eqn: intermediate LS} and \eqref{eqn: intermediate 2 LS}, we have shown that 
\[
\|\na F(\vphi)[\eta]\|_{C^{2,\alpha}(M)}\leq \e \| \na \mathcal{N}^{-1}(\vphi)[\eta]\|_{C^{2,\alpha}}
\]

Now, writing $\na \mathcal{N}^{-1}(\vphi)[\eta] = \eta + \na F(\vphi)[\eta]$ by \eqref{eqn: split}, we see that
\[
\|\na F(\vphi)[\eta]\|_{C^{2,\alpha}(M)} \leq \e \left( \|\eta\|_{C^{2,\alpha}} + \| \na F(\vphi)[\eta]\|_{C^{2,\alpha}}\right).
\]
Absorbing the second term into the left-hand side, and recalling that all norms are equivalent on $K$, we establish \eqref{eqn: estimates for DF}.
This concludes the proof of the lemma.
\end{proof}

\bibliography{references-Yamabe}{}
\bibliographystyle{amsbeta}
	
	%
	%
	%
	%
	%
	%
	%
	%
	%
	%
	%
	%
	%
%

\end{document}